\documentclass[preprint,12pt]{elsarticle}

\usepackage{amssymb}
\usepackage[english]{babel}
\usepackage{graphicx,epstopdf,epsfig}
\usepackage{amsfonts,epsfig,fancyhdr,graphics,amsmath,amssymb}
\usepackage{titlesec}
\usepackage{tikz}
\usepackage{bm}  % Para letras en negritas matemáticas
\usetikzlibrary{matrix}
\usetikzlibrary{calc}
\usepackage{amsthm}
\usepackage{hyperref}
\newtheorem{thm}{Theorem}

\newdefinition{rmk}{Remark}
\newtheorem{defi}{Definition}
\newtheorem{prop}[thm]{Proposition}
\newtheorem{lemm}[thm]{Lemma}
\newtheorem{coro}[thm]{Corollary}
\usepackage{bm}

\begin{document}

\begin{frontmatter}



\title{Roots in the substitution group and in the group of Riordan matrices with ones in the main diagonal}


\author[inst1]{Jorge Calero-Sanz\corref{cor1}}
\ead{jorge.calero@upm.es}
\affiliation[inst1]{organization={Departamento de Matemática Aplicada, Universidad Politécnica de Madrid},%Department and Organization
            addressline={Madrid}, 
            city={Madrid},
            postcode={28040}, 
            state={Madrid},
            country={Spain}}

\author[inst1]{L. Felipe Prieto-Martínez}
\ead{luisfelipe.prieto@upm.es}
\cortext[cor1]{Corresponding author}
\begin{abstract}
We investigate the existence and uniqueness of iterative roots of order $n$ within the substitution group of formal power series $\mathcal J(Z)$ -with coefficients in a commutative ring with unity $Z$-   employing a matrix-based framework grounded in the Riordan group. We analyse the relationship between the substitution group and the Lagrange subgroup - a group of Riordan matrices - and explore some classic questions concerning algebraic completeness and uniqueness of root extractions. This approach allows us to obtain various results about the roots in $\mathcal J(Z)$ for different choices of $Z$. Furthermore, the examination of the substitution group facilitates the analysis of roots within the Riordan matrix group.
\end{abstract}




\end{frontmatter}


\section{Introduction} \label{intro-sec}
The theory of formal power series  has been a subject of significant interest, particularly due to its connections with algebraic and combinatorial structures \cite{Niven}. 

Among these, the \emph{substitution group} of formal power series with coefficients in a commutative ring with unity $Z$, denoted as $\mathcal J(Z)$ and endowed with the composition operation, plays a central role in modeling iterative transformations and related algebraic operations \cite{B,Jen}. A fundamental question in this context is the study of iterative roots: determining when a formal power series admits iterative roots of order $n$ and under what conditions such roots are unique. 

Something similar happens with some groups of \emph{Riordan matrices} with respect to the multiplication. A Riordan matrix is an infinite matrix that we can identify with a pair of formal power series in such a way that the product of two Riordan matrices is related to the multiplication and composition of the corresponding formal power series used for the identification. Recently, the algebraic study of groups of Riordan matrices have received some attention \cite{CK, CJKS, CPT, LMP, LMP2, LMP3} but, as far as we know, a complete study of Riordan matrices admitting roots has not been thoroughly investigated.

This paper deals with some aspects of the algebraic study of the substitution group and groups of Riordan matrices concerning questions about the existence and uniqueness of roots.

\subsection{Main groups of interest}

In the sections that follow, we will always assume that $Z$ denotes a commutative ring with unity and $\mathbb K$ will always denote a field. We denote by $Z[[x]]$ (respectively $\mathbb K[[x]]$) to the set of formal power series $f=f_0+f_1x+f_2x^2+\ldots$ with coefficients in $Z$ (resp. $\mathbb K$).

\medskip

This paper concerns about two groups. The first one is the following:

\begin{defi} Let $Z$ be any commutative ring with unity. The \textbf{substitution group} $\mathcal J(Z)$ is the set of formal power series of the type $g=x+g_2x^2+g_3x^3+\ldots$ in $Z[[x]]$  with the operation of composition. 
\end{defi}

The reader may find more information about this object in \cite{B, Jen, J}.  In the special case $Z=\mathbb F_p$ (where $\mathbb F_p$ denotes the corresponding  finite field) the substitution group is known as the \text{Nottingham group}. More information about this special case can be found in the chapter \emph{The Nottingham Group} by R. Camina in \cite{SSS}.

%\textcolor{red}{CITAR: BABENKO, JOHNSON, NOTTINGHAM GROUP, LIBRO DE NEW HORIZONS PRO-P GROUPS}

\medskip

On the other hand, we have the second object of interest: Riordan matrices. Such objects are infinite matrices with a structure that generalizes Pascal's Arithmetical Triangle. We have decided to present the following definition, which is a generalization, for any commutative ring with unity $Z$, of the original one appearing in \cite{SGWW}. See the notational remark below to put the definition in context with respect to the notation in current use.

\begin{defi}
    Let $Z$ be a commutative ring with unity. A \textbf{Riordan matrix}  is an  invertible infinite matrix of the type
$$\begin{bmatrix} 1\\ d_{10} & 1\\ d_{20} & d_{21} & 1 \\ d_{30} & d_{31} & d_{32} & 1 \\ \vdots & \vdots & \vdots & \vdots & \ddots \end{bmatrix},$$

\noindent  such that there exists two formal power series $f\in (Z[[x]]\setminus xZ[[x]])$ and $g\in\mathcal J(Z)$ satisfying that, for every $j\in\mathbb N$, the generating function of the column $[0, \dots, 0,1,d_{j+1,j},\ldots]^T$ is $f\cdot g^j$, that is,
$$f\cdot g^j=x^j+d_{j+1,j}x^{j+1}+d_{j+2,j}x^{j+2}+\ldots$$

\noindent In this case, we denote the corresponding matrix by $R(f,g)$. \medskip 

We define the \textbf{Riordan group} $\mathcal R'(Z)$ as the set of all Riordan matrices endowed with the multiplication of matrices.
\end{defi}

\medskip

\noindent \textbf{Notational remark.} \emph{If $Z$ is a field $\mathbb K$ of characteristic 0, we do not need ones in the main diagonal for an infinite lower triangular matrix with entries in $Z$ to have an inverse. This leads to a different definition of the concept of Riordan matrix, which is the one in current use \cite{CK, CJKS, CPT, LMP, LMP2, LMP3}, and it is also included herein in Definition \ref{Defi_RK}. The group of these alternative Riordan matrices is denoted by $\mathcal R(\mathbb K)$. The first derivative subgroup of $\mathcal R(\mathbb K)$ is what we have defined here as $\mathcal R'(\mathbb K)$ \cite{LMP}. This has been the motivation for our notation.}

\medskip
%(respectively $\mathcal R_m(\mathbb K)$)   (resp. $\mathcal R_m(\mathbb K)$)

Let us also define $\mathcal R_m'(Z)$  as the matrix group consisting in those $(m+1)\times(m+1)$ matrices that are a principal submatrix of a Riordan matrix. We use the notation $R_m(f,g)$ for the matrix obtained as a principal submatrix of $R(f,g)$. Let us remark that one same element in $\mathcal R_m'(Z)$ can be obtained as a submatrix of infinitely many different Riordan matrices. According to the definition of a Riordan matrix, the initial row and column of a matrix or submatrix begin indexing at $j = 0$.

%\textcolor{red}{DEFINICION ORIGINAL, PERMITE DEFINIR LAS COSAS SOBRE ANILLOS, LA ORIGINAL ES EL CONMUTADOR DEL ACTUAL. LEY DE GRUPO E INVERSA}

%\noindent \textbf{Remark.} In \cite{Jen} the author introduces, for $k\geq 2$, the subgroups $\mathfrak G_k(Z)<\mathcal J(Z)$ consisting of those elements $g\in x+x^kZ[[x]]$. These subgroups $\mathfrak G_k$ of $\mathcal J(Z)$ correspond to the subgroups $\mathcal G_k$ of $\mathcal A'(Z)$ discussed in \cite{LMP}.

\subsection{Roots in groups}

Let $G$ be a group (we will use multiplicative notation) and let $\omega,g\in G$. For some positive integer $n\in\mathbb N$, we say that $\omega$ is a \textbf{root of order $n$} of $g$ if $\omega^n=g$. A \emph{root of order 2 of $g$} is usually referred just as a \emph{root of $g$}.

\medskip

If the operation in $G$ is the composition, as is the case of the group $\mathcal J(Z)$, then the roots are sometimes called \emph{iterative roots} or \emph{fractional iterates}. In this context, the notation $\omega^{[n]}$ is more common for the iterated composition than $\omega^n$. In Analysis, the study of iterative roots, in some groups of functions (typically continuous) endowed with the composition, is of great interest. This has also motivated the study of iterative roots in $\mathcal J(Z)$, usually for the cases $Z=\mathbb R,\mathbb C$. The reader may find more information in the survey \cite{Ja} and in the references therein. In order to characterise a group by the number of roots that an arbitrary element can have, we present the following definitions, which can be found in  \cite{B}:
%\noindent \textbf{Remark.} \emph{Note that the problem of deciding the existence of roots of order $n$ of a given element $g\in\mathcal J(Z)$ is related to the problem of finding a canonical representative with respect to conjugacy for each element in $\mathcal J(Z)$ (see \cite{S}). However, this is not the approach followed in this paper.
%}

%\medskip

\begin{defi}
We say that: \\

$G$ is  a group with \textbf{unique root extraction} (URE-group) if,  for any $g\in G$ and for any $n\in\mathbb N$, $g$ has at most one \textbf{root}, that is, the equation $\omega^n=g$ has at most one solution $\omega\in  G$, for each $g\in G$; \bigskip

$G$ is an \textbf{algebraically complete} group or \textbf{divisible} group if the equation $\omega^n=g$ has solutions for every $g\in G$  and for every $n\in\mathbb N$.
\end{defi}
\medskip

For instance, the group $\mathcal J(\mathbb Z)$ is an URE-group, whereas the groups $\mathcal J(\mathbb Q)$, $\mathcal J(\mathbb R)$ are algebraically complete (proposition 5.3 in \cite{B}). 

\subsection{What is it done in this paper?}

\medskip

%We prove the following:

%\begin{itemize}

%\item First, we give a Riordan group-type argument ( \textcolor{red}{Pensar este paréntesis) }not relying in finding canonical representative with respect to conjugation)  to find a neccessary and sufficient set of equations that a formal power series $g\in \mathcal J(Z)$ should satisfy in order to have a root of order $n$.

%\item After this, we use these equations to prove that $\mathcal J(\mathbb K)$ is an algebraically complete URE-group, for every field $\mathbb K$ of characteristic 0. Moreover, we check that $\mathcal J(Z)$ is not algebraically complete if $Z$ has positive characteristic and that $\mathcal J(\mathbb Z)$ is an URE-group but it is not algebraically closed. \textcolor{red}{Expandir secciones 2 y 3}

%\item In Section \ref{Section_RZ} we exhibit the relation between questions of existence and uniqueness of roots of order $n$ in $\mathcal R'(Z)$  and questions of existence and uniqueness of eigenvectors of eigenvalue 1 of elements in $\mathcal R'(Z)$.

%\item We use the connection made in Section \ref{Section_RZexamples} and some results in \cite{LMP} to prove that $\mathcal R'(\mathbb K)$ is an algebraically complete URE-group and that $\mathcal R'(\mathbb Z)$ is an URE-group. \textcolor{red}{Sección 5 por 4} \textcolor{red}{Putting all together}

%\item Finally, in the last section, we briefly comment some connection between the results appearing in this paper and some other problems concerning the algebraic structure of $\mathcal J(Z)$ and $\mathcal R'(Z)$.
%\end{itemize}

The document is structured as follows: In Section \ref{Section_JZRiordan}, we give a Riordan group-type argument to find a necessary and sufficient set of equations that a formal power series $g\in \mathcal J(Z)$ should satisfy in order to have a root of order $n$. After this, in Section \ref{Section_JZ}, we use these equations to prove that $\mathcal J(\mathbb K)$ is an algebraically complete URE-group, for every field $\mathbb K$ of characteristic 0. Moreover, we check that $\mathcal J(Z)$ is not algebraically complete if $Z$ has positive characteristic and that $\mathcal J(\mathbb Z)$ is an URE-group but it is not algebraically closed. In Section \ref{Section_RZ}, we exhibit the relation between questions of existence and uniqueness of roots of order $n$ in $\mathcal R'(Z)$  and questions of existence and uniqueness of eigenvectors of eigenvalue 1 of elements in $\mathcal R'(Z)$. In Section \ref{Section_RZexamples}, we use the connections established in previous sections and some results in \cite{LMP} to prove that $\mathcal R'(\mathbb K)$, for every field $\mathbb K$ of characteristic $0$, is an algebraically complete URE-group and that $\mathcal R'(\mathbb Z)$ is an URE-group. In Section \ref{SectionConnections}, we conclude and we briefly comment some connections between the results appearing in this paper and some other problems concerning the algebraic structure of $\mathcal J(Z)$ and $\mathcal R'(Z)$.

%\item If $\mathbb K$ is a field of characteristic 0, then $\mathcal J(\mathbb K)$ (and so $\mathcal A'(\mathbb K)$) is an algebraically complete URE-group.

%\item In general,  for a field $\mathbb K$ of positive charactersitic,  $\mathcal J(\mathbb K)$ needs not to be an URE-group, neither an algebraically closed group. In this case, the characteristic may suffice to have a criterion for the existence of roots of a given order.

%\item $\mathcal J(\mathbb Z)$ is an URE-group, but fails to be algebraically closed. Moreover, the set of elements in $\mathcal J(\mathbb Z)$ admitting an iterative square root is, in some sense, ``small''. We also discuss some conditions related to the existence of square roots in $\mathcal J(\mathbb Z)$.

%\item If $\mathbb K$ is a field of characteristic 0, then the Riordan group $\mathcal R'(\mathbb K)$ is an algebraically complete URE-group. \textcolor{red}{(INCOMPLETO)}

%\item Some applications and related problems

%\end{itemize}

%%%%%%%%%%%%%%%%%%%%%%%%%%%%%%%%%%%%%%%%%%%%%%
%\section{Basic questions about roots in $\mathcal J(Z)$}

\section{Characterization of roots in \texorpdfstring{$\mathcal J(Z)$}{J(Z)} via Riordan matrices}\label{Section_JZRiordan}

As a start, we will explore the relationship between the substitution group $\mathcal J(Z)$ and the Riordan group. Our purpose is to connect the study of roots in the substitution group to the study of roots in the Riordan group, where some iterative processes are available to decide their existence and uniqueness and, if possible, to iteratively compute them. Let us begin introducing an important subgroup of the Riordan group:

\begin{defi}
The \textbf{Lagrange subgroup}, also called \textbf{associated subgroup}, $\mathcal L'(Z)$ of the Riordan group $\mathcal R'(Z)$ consists in those Riordan matrices of the type $R(1,g)$. 

 In the same manner, we define the Lagrange subgroup $\mathcal L_m'(Z)$ of  $\mathcal R_m'(Z)$ as the finite Riordan matrices of the type $R_m(1,g)$.
\end{defi}

\medskip

\noindent \textbf{Remark.} \emph{Note that each element $\omega=x+\omega_2x^2+\omega_3x^3+\ldots$ in $\mathcal J(Z)$ can be identified with the element }
\begin{equation} \label{eq.ass} \resizebox{.9\textwidth}{!}{$   R(1,\omega) = \left[\begin{array}{lllllllll} 1\\ 0 & 1  \\ 0 &\omega_2 & 1 \\0&\omega_3 & 2\omega_2& 1\\0& \omega_4&2\omega_3+\omega_2^2 & 3\omega_2& 1\\0& \omega_5&2\omega_4+2\omega_2\omega_3 & 3\omega_3+3\omega_2^2& 4\omega_2& 1\\ 0& \omega_6& 2\omega_5+2\omega_2\omega_4+\omega_3^2& 3\omega_4+6\omega_2\omega_3+\omega_2^3& 4\omega_3+6\omega_2^2& 5\omega_2&1 & \\
 0& \omega_7 & 2\omega_6+2\omega_2\omega_5+2\omega_3\omega_4 & 3\omega_5+6\omega_2\omega_4+3\omega_3^2+3\omega_2^2\omega_3 &4\omega_4+12\omega_2\omega_3+4\omega_2^3 &5\omega_3+10\omega_2^2 & 6\omega_2 & 1\\
 \vdots & \vdots&\vdots &\vdots & \vdots& \vdots&  \vdots&\vdots  &  \ddots \end{array}\right] $}\end{equation}

\noindent  \emph{in $\mathcal L'(Z)$ and vice-versa. Therefore, the groups $\mathcal J(Z)$ and $\mathcal L'(Z)$ are isomorphic. However, the natural identification $\omega\mapsto R(1,\omega)$ is not the one that induces isomorphism, since, in general, $$R(1,g)\cdot R(1,h) = R(1,h\circ g) \neq R(1,g\circ h).$$ In fact, it is $\omega\mapsto R(1,\omega)^{-1}$ that induces this identification (formally, the natural identification is denoted as anti-isomorphism).}
\medskip

The identification $\omega\mapsto R(1,\omega)^{-1}$  is really useful, since it allow us to use matricial notation in the proofs and formulas for results concerning roots of elements in $\mathcal J(Z)$.

\begin{thm} \label{theo.bf} Let $Z$ be any commutative ring with unity and $g=x+g_2x^2+g_3x^3+\ldots$ and $\omega=x+\omega_2x^2+\omega_3x^3+\ldots$ be two elements in $\mathcal J(Z)$. Let us denote $\bm c_k=[\omega_2,\ldots, \omega_{k-1}]^T$ and $\bm r_k$ to the the row vector consisting in the elements in the positions $(k,2),\ldots, (k,k-1)$ in $R_k(1,\omega)$.

\begin{itemize}

\item  (Case $n=2$) The formal power series $\omega$ is a solution of the equation $\omega\circ\omega=g$ if and only if

\begin{align}
2\omega_2 &= g_2, \tag{2.1a} \label{eq.heaven2} \\
2\omega_k &= g_k - \bm{r}_k \bm{c}_k, \quad \text{for every } k \geq 3. \tag{2.1b} \label{eq.hell2}
\end{align}

%\begin{equation}\label{eq.heaven2}2\omega_2=g_2 \end{equation}

%%\begin{equation}\label{eq.hell2} 2\omega_2=g_2,\qquad 2\omega_k =g_k-\bm r_k\bm c_k,\end{equation}

%\noindent where $\bm M_k=R_{k-3}(\left(\frac{\omega}{x}\right)^2,\omega)$, 

\item The formal power series $\omega$ is a solution of the equation $\omega^{[n]}=g$, for some fixed integer $n\geq 3$, if and only if

\begin{align}
n\omega_2 & =g_2, \tag{2.1c}\label{eq.heaven}\\
n\omega_k & =g_k-\bm r_k(\bm M_k^{n-2}+2\bm M_k^{n-3}+\ldots+(n-3)\bm M_k+ (n-2)\bm I)\bm c_k, \quad \text{for every } k \geq 3,  \tag{2.1d}\label{eq.hell}
\end{align}

%\begin{align}
 %n\omega_2=g_2 \label{eq.heaven}\tag{2.3a}\\
 %n\omega_k =g_k-\bm r_k(\bm M_k^{n-2}+2\bm M_k^{n-3}+\ldots+(n-3)\bm M_k+ (n-2)\bm I)\bm c_k, &  \label{eq.hell}\tag{2.3b} 
%\end{align}

\noindent where $\bm M_k=R_{k-3}(\left(\frac{\omega}{x}\right)^2,\omega)$

%, $\bm c_k=[\omega_2,\ldots, \omega_{k-1}]^T$ and $\bm r_k$ is the row vector consisting in the elements in the positions $(k,2),\ldots, (k,k-1)$ in $R_k(1,\omega)$.
\end{itemize}

\end{thm}

\begin{proof} Let us demonstrate the general case $n \geq 3$ and let us omit the case $n=2$. Under the previously outlined assumptions, note that $\omega^{[n]}=g$ if and only if $(R_m(1,\omega))^n=R_m(1,g)$, for every positive integer $m\geq 2$. In turn, we present the following statement:

\medskip

\noindent \textbf{Statement.} For $m\geq 2$, $R_m(1,\omega))^n=R_m(1,g)$ if and only if Equation \eqref{eq.heaven} holds and Equation \eqref{eq.hell} holds, for every $3\leq k\leq m$.

\medskip

Therefore, this statement will be proven by induction on $m$, which will suffice to demonstrate the overall result.

\noindent For the case $m=2$, we only need to see that
$$(R_2(1,\omega))^n=\begin{bmatrix} 1 \\ 0 & 1 \\ 0 & \omega_2 & 1 \end{bmatrix}^n=\begin{bmatrix}  1\\ 0 & 1 \\ 0 & n\omega_2 & 1 \end{bmatrix}=\begin{bmatrix}  1\\ 0 & 1 \\ 0 & g_2 & 1 \end{bmatrix}=R_2(1,g) $$

\noindent and this happens if and only if the equation \eqref{eq.heaven} holds.

\medskip

\noindent Now, let us assume that the result is true for every integer smaller than $k$, with $2\leq k < m$. Let us prove that it is true for $m$. Note that 
$$(R_m(1,\omega))^n=\left[\begin{array}{c  | c} &\\ (R_{m-1}(1,\omega)) &  \\ & \\ \hline 0 \mid \omega_m \mid\bm r_m  & 1  \end{array}\right]^n,$$

\noindent \noindent where $\bm r_m$ is the corresponding row vector of length $m-2$.  Using block multiplication and the induction assumption, we get that all the entries in $(R_m(1,\omega))^n$ are equal to those in $R_m(1,g)$ except for, perhaps, those in the last row. On the other hand, we have that

$$R_{m-1}(1,\omega)=\left[\begin{array}{ c| c| c }1& &  \\ \hline 0 & 1 &   \\ \hline \bm 0 & \bm c_m  & \bm M_m   \end{array}\right],$$

\noindent \noindent where $\bm 0, \bm c_m$ are the corresponding column vectors of length $m-2$ and $\bm M_m$ is the submatrix obtained by eliminating the initial two rows and the initial two columns. Note that for a Riordan matrix $R(f,g)$, the operation of deleting the 0-row and the 0-column yields the Riordan matrix $R(f \cdot \frac{g}{x},g)$. Hence, we obtain that $\bm M_m=R_{m-3}(\left(\frac{\omega}{x}\right)^2,\omega)$. Using block multiplication we get:
\[
\vcenter{\hbox{$(R_m(1, \omega))^n =$}}
\begin{tikzpicture}[baseline=(m.center)] 
    \matrix (m) [matrix of math nodes, left delimiter={[}, right delimiter={]},
    nodes={minimum width=0.2cm, minimum height=0.2cm,text centered}, row sep=0.35em, column sep=0.35em] {
        1   &  \textcolor{white}{\bm{c_m}}    &  \textcolor{white}{0}     & \textcolor{white}{0} \\
        0   &  1   &  \textcolor{white}{0}     & \textcolor{white}{0} \\
        \bm{0} & \bm{c_m} & \bm{M_m} & \textcolor{white}{0} \\
        0   & \omega_m & \bm{r_m} & 1 \\
    };
    \draw[thick, red] 
       ($(m-1-2.north west)+(-0.1,0.2)$) -- 
        ($(m-1-2.north east)+(0.1,0.2)$) -- 
        ($(m-4-2.south east)+(0.1,-0.2)$) -- 
        ($(m-4-2.south west)+(-0.1,-0.2)$) -- cycle;
      \draw[thick, red] 
       ($(m-4-1.north west)+(-0.1,0.1)$) -- 
        ($(m-4-4.north east)+(0.1,0.1)$) -- 
        ($(m-4-4.south east)+(0.1,-0.25)$) -- 
        ($(m-4-1.south west)+(-0.1,-0.25)$) -- cycle;
\foreach \X [evaluate=\X as \Y using {int(\X+1)}]in {1,2,3}
{\path ($(m-\X-1.south west)!0.5!(m-\Y-1.north west)$) coordinate
(aux\X);
\draw[black,dashed] (m.west |- aux\X) -- (m.east |- aux\X);
}
\foreach \X [evaluate=\X as \Y using {int(\X+1)}]in {1,2,3}
{\path ($(m-1-\X.east)!0.5!(m-1-\Y.west)$) coordinate
(auy\X);
\draw[black,dashed] (m.north -| auy\X) -- (m.south -| auy\X);
}
\end{tikzpicture}^n =
\]

$$=\left[\begin{array}{ c| c| c |c}1& &  \\ \hline 0 & 1 & &  \\ \hline \bm 0 & (\bm M_m^{n-1}+\bm M_m^{n-2}+\ldots+\bm I)\bm c_m  & \bm M_m^n & \\
\hline  0& n\omega_m+\bm r_m(\bm M_m^{n-2}+2\bm M_m^{n-3}+\ldots+ (n-1)\bm I)\bm c_m& \bm r_m(\bm M_m^{n-1}+\bm M_m^{n-2}+\ldots+\bm I)& 1 \end{array}\right],$$

\noindent where $\bm I$ is the identity matrix of size $(m-2)\times (m-2)$. To elucidate this matrix multiplication, it is evident that the block in position $(2,2)$ in the r.h.s in the previous equation is $\bm M_{m}^{n}$. This allows for the computation of the blocks at positions $(2,3)$ and $(3,2)$, leading to the final calculation of the block at position $(3,1)$ marked in red in the previous equation. 

So we can see that $(R_m(1,\omega))^n=R_m(1,g)$ if and only if Equation \eqref{eq.hell} holds for $k=m$, since the last row of an arbitrary element $R_m(1,h)$ in $\mathcal L'(Z)$ is totally determined by the element in the position $(m,1)$ and $R_{m-1}(1,h)$. 
\end{proof}

\medskip

\noindent \textbf{Remark.} \emph{Note that all the terms in the right hand side of Equation \eqref{eq.hell} (same for \eqref{eq.hell2})  depend only on $\omega_2,\ldots,\omega_{k-1}$ and on the coefficients of $g$. So Equation \eqref{eq.hell} allow us to iteratively compute the coefficients $\omega_2,\omega_3,\omega_4,\ldots$ from the coefficients of $g$.}

\medskip

The reader may see that, using Equation \eqref{eq.heaven2} and Equation \eqref{eq.hell2}, one can easily recover the first terms of the sequence $(a_0,a_1,a_2,\ldots)$, labeled A098932 in the \emph{Online Encyclopedia of Integer Sequence} \cite{OEIS}, which corresponds to the formal power series
$$\omega(x)=\frac{a_0}{2^0\cdot 1!}x+\frac{a_1}{2\cdot 3!}x^3+\frac{a_2}{2^2\cdot 5!}x^5+\ldots $$

\noindent that satisfies $\omega\circ\omega=\sin x$. In general, one can find many other sequences corresponding to the iterative root of other formal power series (e.g. $\tan x$, $x/(1-x)^2$, $e^x-1$, \ldots). The same happens for iterative roots of order greater than 2.

\medskip

To conclude this section, we deveal a relationship between the order of a root and its multiplicity. Following the terminology in \cite{M}, for an element $g=x+g_2x^2+g_3x^3+\ldots \in\mathcal J(Z)$ different from the identity, we define the \textbf{multiplicity} to be the minimum number $k\geq 2$, such that $g_k\neq 0$. We can say that the identity $g=x$ has multiplicity equal to $\infty$. Then, we have the following:

\begin{prop} \label{Prop2.2} Let $g,\omega\in\mathcal J(Z)$ and let us suppose that $\omega$ is a root of order $n$ of $g$. Then the multiplicity of $\omega$ is less or equal to the multiplicity of $g$.

\end{prop}

\begin{proof} This can be easily proved via the identification of $\mathcal J(Z)$ and $\mathcal L'(Z)$. 

As we have that $\omega^{[n]} = g$ if and only if $(R(1,\omega))^n = R(1,g)$, which, in turn, implies that $(R_{m}(1,\omega))^n=R_{m}(1,g)$, $\forall m \geq 2$.

Hence, note that the multiplicity of $\omega$ is $k$ if and only if $R_{k-1}(1,\omega)$ is the $k\times k$ identity matrix $\bm I$. So, if $\omega$ is a root of order $n$ of $g$, $(R_{k-1}(1,\omega))^n=R_{k-1}(1,g)=\bm I$, which means that the multiplicity of $g$ is $k$ or greater.
\end{proof}

%%%%%%%%%%%%%%%%%%%%%%%%%%%%%%%%%%%%%%%%%%%%%%%%%%%%%%%%%%%%%%%%%%%%
\section{Roots in \texorpdfstring{$\mathcal J(Z)$}{J(Z)} for some particular examples of the ring \texorpdfstring{$Z$}{Z}}\label{Section_JZ}

Section \ref{Section_JZRiordan} has established a connection between the power series in $\mathcal J(Z)$ and the Lagrange subgroup $\mathcal L'(Z)$. Furthermore, Theorem \ref{theo.bf} has enabled us to describe the roots of order $n$ via a system of equations. Let us now examine the implications of these results when we impose additional 
conditions on the ring \( Z \).

\subsection{Roots in \texorpdfstring{$\mathcal J(\mathbb K)$}{J(K)}, for \texorpdfstring{$\mathbb K$}{K}   being a field of characteristic \texorpdfstring{$0$}{0}}

If the ring $Z$ is additionally a field $\mathbb K$ of characteristic $0$, the system of equations in Theorem \ref{theo.bf} has a unique solution. Hence, we have the following result:

\begin{thm} \label{theo.bg} For every field $\mathbb K$ of characteristic 0, $\mathcal J(\mathbb K)$ is an algebraically complete URE-group.

\end{thm}

\begin{proof} As we have already said, under the assumptions described above,  $\omega^{[n]}=g$ if and only if, for every positive integer $m\geq 2$, $R_m(1,\omega))^n=R_m(1,g)$. According to Theorem \ref{theo.bf}, this happens if and only if  Equation \eqref{eq.heaven} holds and Equation \eqref{eq.hell} holds, for every $k\geq 3$. Further, proving that $\mathcal J(\mathbb K)$ is algebraically complete and an URE-group means the existence of a unique $\omega$ that fulfils the equation $\omega^{[n]}=g$. Consequently, demonstrating the theorem is equivalent to verifying the following statement:

\medskip

\noindent \textbf{Statement.} For all $m \geq 2$, the linear system $S_m$ consisting in the $m-1$ equations established in  Eq. \eqref{eq.heaven} and Eq. \eqref{eq.hell}, for every $k$ with  $3\leq k\leq m$, has a unique solution in the indeterminates $\omega_2,\ldots,\omega_m$.

\medskip

\noindent Hence, we will prove the statement by induction over $m$:

\medskip

In the case $m=2$, we only need to care about Equation \eqref{eq.heaven} an this equation has, obviously, only one solution in the indeterminate $\omega_2$.  \\

For $m\geq 3$, suppose that the statement is true for $m-1$. The linear system $S_m$ is obtained from $S_{m-1}$ just adding the linear equation
$$n\omega_m =g_m-\bm r_m(\bm M_m^{n-1}+2\bm M_m^{n-2}+\ldots+ (n-1)\bm I)\bm c_m.$$

\noindent The indeterminates $\omega_2,\ldots, \omega_{m-1}$ are determined by $S_{m-1}$ by induction hypothesis. It is important to observe that the submatrix $\bm M_m$ exclusively involves the indeterminates $\omega_2,\ldots,\omega_{m-2}$. Considering that $\bm r_m$ represents the final row of this submatrix $\bm M_m$, and based on the definition of $\bm c_m$, it is clear that the unique indeterminates presented on the r.h.s. in the aforementioned equation are $\omega_2,\ldots,\omega_{m-1}$. We can conclude that $S_m$ has a unique solution, concluding the demonstration.
$ $
\end{proof}

Furthermore, in a field of characteristic zero, if $\omega$ is the $n$-th root of $g$, the multiplicities of both power formal series are equal: 

\begin{prop} For every element $g\in \mathcal J(\mathbb K)$  and for every $n\geq 2$, if $\omega$ is the root of order $n$ of $g$, then 
the multiplicity of $g$ equals the multiplicity of $\omega$.

\end{prop}

\begin{proof} Let us define the multiplicity of $\omega$ and $g$ as $k$ and $m$, respectively. Proposition \ref{Prop2.2} give us that the multiplicity of $k$ is less or equal to the multiplicity of $g$. In addition, it has been established in the same proposition that the multiplicity of $\omega$ is $k$ if and only if $R_{k-1}(1,\omega)$ coincides with the identity matrix of size $k \times k$.

Note that $R_k(1,\omega)$ is a matrix that coincides with the identity matrix of size $(k+1)\times(k+1)$ in all the entries except for the one in the position $(k+1,1)$, where the entry equals $\omega_k\neq 0$. Then $(R_k(1,\omega))^n$ is a matrix that coincides with the identity matrix of size $(k+1)\times(k+1)$ in all the entries except for the one in the position $(k+1,1)$, where the entry equals $n\omega_k\neq 0$. Hence, $\omega$ and $g$ have the same multiplicity.
\end{proof}

%%%%%%%%%%%%%%%%%%%%%%%%%%%%%%%%%%%%%%%%%%%%%%%%%%%%%%%%%%%%%%%%%%%%
%\subsection{Results about roots in the substitution group \texorpdfstring{$\mathcal J(Z)$}{J(Z)}, for \texorpdfstring{$Z$}{J(Z)} being a ring (or field) of positive characteristic.}

\subsection{Roots in  \texorpdfstring{$\mathcal J(Z)$}{J(Z)}, for \texorpdfstring{$Z$}{J(Z)} being a general ring of positive characteristic.}

\medskip

Despite fields of characteristic zero, a general ring $Z$ of characteristic $p>0$ makes that $\mathcal J(Z)$ does not exhibit the properties of being URE-groups and algebraically complete. To shed light on this aspect, we can observe that the set of geometric power series 
$$\{x + rx^2 + r^2 x^3 + \dots + r^k x^{k+1} + \dots =\frac{x}{1-rx}\in \mathcal J(Z):r\in Z\}$$ form a subgroup that we shall refer to as the \textbf{geometric subgroup}, denoted by $\mathcal G(Z)$. The coefficients of every element in the geometric subgroup form a geometric sequence with first term 1 and rate $r$. 

\noindent although it does not make so much sense for general rings.

The following lemma demonstrates that $\mathcal G(Z)$ is a subgroup of $\mathcal J(Z)$ isomorph to $Z$: 

\begin{lemm} For every ring with unity $Z$, $\mathcal G(Z)$ is a subgroup of $\mathcal J(Z)$. Moreover, the identification
$$r\longmapsto\frac{x}{1-rx}$$

\noindent is an isomorphism between $Z$ and $\mathcal G(Z)$   (endowed with its sum operation).
\end{lemm}

\begin{proof} The identification presented is bijective. So we only need to check that, for $g=\frac{x}{1-rx}$, $h=\frac{x}{1-sx}$, we have that
$$g\circ h=\frac{\frac{x}{1-sx}}{1-r\frac{x}{1-sx}}=\frac{x}{1-(r+s)x}.$$
$ $ \end{proof}
%, $\overline{g}=\frac{x}{1+rx}$, 

Through the lemma, we are able to show that $\mathcal J(Z)$ is neither an URE-group nor algebraically complete.

\begin{thm} Let $Z$ be a ring of characteristic $p>0$.

\begin{itemize}

\item[(a)]  There exist elements of  finite order $p$ in $\mathcal J(Z)$, aside from $g=x$.

\item[(b)] For every integer $n$ with $n\equiv 0\mod p$, there exists some element $g\in\mathcal J(Z)$ such that $g$ has more than one root $\omega$ of order $n$. Therefore, $\mathcal J(Z)$ is not an URE-group.

\item[(c)] For every integer $n$ with $n\equiv 0\mod p$, there exists some element $g\in \mathcal J(Z)$ such that $g$ has no root $\omega$ of order $n$. In consecuence, $\mathcal J(Z)$ is not algebraically complete.

\end{itemize}

\end{thm}

\begin{proof}  For Statement (a), let us denote $g=x$ and $\omega=\frac{x}{1-x}\in\mathcal G(Z)$. Then, $g^{[p]} = x$  and, as a consequence of the previous lemma, $\omega^{[p]} = \frac{x}{1 - px} = x$, i.e, $g$ and $\omega$ are both roots of order $p$ of $g$.

\medskip

Statement (b) is a consequence of Statement (a).

\medskip 

For Statement (c), let us consider $n$ with $n\equiv 0\mod p$ and $g=x+x^2$. If $\omega=x+\omega_2x^2+\omega_3x^3+\ldots$ were a root of order $n$ of $g$, it must fulfill:

$$(R_2(1,\omega))^n=
\begin{bmatrix}1\\ 0 & 1 \\ 0 & \omega_2 & 1 \end{bmatrix}^n
=\begin{bmatrix}1\\ 0 & 1 \\ 0 & n \omega_2 & 1 \end{bmatrix}
=\begin{bmatrix}1\\ 0 & 1 \\ 0 & 0 & 1 \end{bmatrix} \neq \begin{bmatrix}1\\ 0 & 1 \\ 0 & 1 & 1 \end{bmatrix} = (R_2(1,g)).$$

\noindent  So $g$  has no roots of order $n$.
\end{proof}

Let us deal now with the particular case of $Z$ being a field (of positive characteristic). First, we have the following result,  which has the same prove as Theorem \ref{theo.bf}:

\begin{coro} Let $\mathbb K$ be a field of characteristic $p>0$. For every integer $n$, $n\not\equiv 0\mod p$, for every $g\in\mathcal J(\mathbb K)$, $g$ has a unique root $\omega$ of order $n$.

\end{coro}

In order to understand better the behaviour of the groups $\mathcal J(Z)$, let us briefly explore the case $Z=\mathbb Z_2$, as a particular example. The subsequent proposition illustrates that for the simplest ring, specifically \( Z = \mathbb{Z}_2 \), the characterisation of roots in truncated power series is complex; thus, generalising this study will be more complicated.

\begin{prop}
    
 Let $g=x+g_2x^2+g_3x^3+\ldots\in\mathcal J(\mathbb Z_2)$, then  $R_{7}(1,g)$ admits a square root $R_{7}(1,\omega)$ in $\mathcal{L}_{7}'(\mathbb Z_2)$  if and only if $g_2=g_3=0$ and we are in one of the following cases:

\begin{itemize}

\item  $g_4=0$, $g_5=g_7=0$,

\item $g_4=0$, $g_5=g_7=1$,

\item  $g_4=1$, $g_5=g_7=1$.

\end{itemize}

\end{prop}

\begin{proof} 
In view of the matrix in Equation \eqref{eq.ass}, it is clear that the matrices $R_{7}(1,\omega) \in \mathcal L_7'(\mathbb Z_2)$ are of the type
\begin{eqnarray*} 
\underbrace{\left[\begin{array}{l l l l l l l l} 1& & & & & & & \\ 0& 1& & & & & & \\
0&\omega_2 &1 & & & & & \\
0&\omega_3 &0 &1 & & & & \\
0& \omega_4&\omega_2 & \omega_2&1 & & & \\
0&\omega_5 &0 &\omega_2+\omega_3 & 0&1 & & \\
0&\omega_6 & \omega_3&\omega_2+\omega_4 & 0 & \omega_2&1 & \\
0&\omega_7 &0 &\omega_3+\omega_5+\omega_2\omega_3 &0 &\omega_3 & 0&1 \\
 \end{array}\right]}_{R_{7}(1,\omega)} \Rightarrow 
\underbrace{\left[\begin{array}{l l l l l l l l} 
1& & & & & & & \\ 
0& 1& & & & & & \\
0& 0 &1 & & & & & \\
0&0 &0 &1 & & & & \\
0& \omega_{2} + \omega_2 \omega_3& 0 & 0&1 & & & \\
0&\omega_{3} + \omega_2 \omega_3 &0 &0 & 0&1 & & \\
0& \omega_3 \omega_4 + \omega_2 \omega_5 & 0 & \omega_{2} + \omega_2 \omega_3 & 0 & 0&1 & \\
0&\omega_{3} + \omega_2 \omega_{3}&0 & \omega_{3} + \omega_{2}\omega_{3} &0 &0 & 0&1 \\
\end{array}\right]}_{R_{7}(1,\omega)^2}  .\end{eqnarray*} 

\noindent Imposing that the elements in the 1-column of $(R_7(1,\omega))^2$ equals the elements in the 1-column of $R_7(1,g)$, where we have clearly that $g_2 = g_3 = 0$, and we need that:
$$\begin{cases}\omega_2+\omega_2\omega_3=g_4\\ 
\omega_3+\omega_2\omega_3=g_5 \\
\omega_2\omega_5+\omega_3\omega_4=g_6 \\
\omega_3+\omega_2\omega_3=g_7\end{cases} \Longrightarrow \begin{cases}\omega_2+\omega_2\omega_3=g_4\\ 
\omega_3+\omega_2\omega_3=g_5=g_7 \\
\omega_2\omega_5+\omega_3\omega_4=g_6 \end{cases}$$

\noindent Now, let us do a case by case study of this system:
\begin{itemize}
\item $\omega_2=0,\omega_3=0\Rightarrow g_4=g_5=g_6=g_7=0$,

\item $\omega_2=0,\omega_3=1\Rightarrow g_4=0$, $g_5=g_7=1$,

\item $\omega_2=1,\omega_3=0\Rightarrow g_4=1$, $g_5=g_7=0$,

\item $\omega_2=1,\omega_3=1\Rightarrow g_4=g_5=g_7=0$.

\end{itemize}
$ $
\end{proof}

%%%%%%%%%%%%%%%%%%%%%%%%%%%%%%%%%%%%%%%%%%%%%%%%%%%%%%%%%%%%%%%%%%%%%%%%%%

\subsection{Roots in \texorpdfstring{$\mathcal J(\mathbb Z)$}{J(Z)}}

The study of roots in $\mathcal{J}(Z)$ where $Z = \mathbb Z$ is well-known in the literature \cite{B}, as explained in the introduction. We have the following result: 

\begin{prop}
$\mathcal J(\mathbb Z)$ is an URE-group and not algebraically closed.    
\end{prop}
\begin{proof} \noindent $\mathcal J(\mathbb Z)$ is a subgroup of $\mathcal J(\mathbb Q)$. As a consequence, it is immediate that $\mathcal J(\mathbb Z)$ is an URE-group. \\

\medskip

\noindent There exists a surjective homomorhism from $\mathcal J(\mathbb Z)$ to $\mathcal J(\mathbb Z_p)$, for every positive integer $p$, given by
$$x+\omega_2x^2+\omega_3x^+\ldots\longmapsto x+\overline\omega_2x^2+\overline\omega_3x^3+\ldots,\;\text{ where, }\forall i\geq 2,\;  \overline\omega_i\equiv \omega_i\mod p.$$

\noindent As a consequence, $\mathcal J(\mathbb Z)$ is not algebraically closed. Moreover, for every positive integer $p$, there exists an element  $g\in\mathcal J(\mathbb Z)$ such that $g$ has no root of order $p$.
\end{proof}

%\textbf{Remark.} \emph{$\mathcal J(\mathbb Z)$ is a subgroup of $\mathcal J(\mathbb Q)$. As a consequence, it is immediate that $\mathcal J(\mathbb Z)$ is an URE-group (which is well known in the literature, as explained in the introduction).}

%\medskip

%\noindent \textbf{Remark.} \emph{There exists a surjective homomorhism from $\mathcal J(\mathbb Z)$ to $\mathcal J(\mathbb Z_p)$, for every positive integer $p$, given by}
%$$x+\omega_2x^2+\omega_3x^+\ldots\longmapsto x+\overline\omega_2x^2+\overline\omega_3x^3+\ldots,\;\text{ where, }\forall i\geq 2,\;  \overline\omega_i\equiv \omega_i\mod p.$$

%=\begin{cases}0&\text{if $\omega$ is even}\\ 1 & \text{if $\omega_i$ is odd}\end{cases}.$$

%\noindent \emph{As a consequence, $\mathcal J(\mathbb Z)$ is not algebraically closed. Moreover, for every positive integer $p$, there exists an element  $g\in\mathcal J(\mathbb Z)$ such that $g$ has no root of order $p$.}

%\medskip

% The substitution group $\mathcal J(\mathbb Z)$ is not algebraically closed. Let $g=x+x^2\in\mathcal J(\mathbb Z)$. Note that the equation $\omega\circ\omega=g$ has no solutions. For any possible $R_2(1,\omega)$, we have that
%$$\begin{bmatrix}  1 \\ 0 & 1 \\ 0 & \omega_2 & 1 \end{bmatrix}^2 =\begin{bmatrix} 1 \\ 0 & 1 \\ 0 & 2\omega_2 & 1 \end{bmatrix}$$

%\noindent and the matrix $R_2(1,g)$ has an odd number in the position $(2,1)$.

The preceding proposition indicates that the existence of square roots depends on the distribution of even and odd integers among the coefficients of the power series. Therefore, we have the subsequent:

\begin{prop} Let $g=x+g_2x^2+g_3x^3+\ldots\in\mathcal J(\mathbb Z)$ and let $R_k(1,g)\in\mathcal L_k'(\mathbb Z)$ a matrix that admits a square root. 
Let $\widetilde g=x+g_2x^2+\ldots+g_kx^k+\widetilde g_{k+1}x^{k+1}+\ldots$ be another element in $\mathcal J(\mathbb Z)$ that coincides with $g$ up to the coefficient of the power $x^k$ (and so $R_k(1,g)=R_k(1,\widetilde g)$). There exists a choice of the parity of $\widetilde g_{k+1}$ for which $R_{k+1}(1,\widetilde g)$ admits a square root and for the opposite choice of the parity $R_{k+1}(1,\widetilde g)$ does not admits a square root.
\end{prop}

\begin{proof} Since $R_k(1,\widetilde g)$ admits a square root $R_k(1,\omega)$, to make $(R_{k+1}(1,\omega))^2=R_{k+1}(1,\widetilde g)$ we only need to ensure that the entry in the position $(k+1,1)$ is the adequate in both matrices. The equation for this position is
$$2\omega_{k+1}+\alpha(\omega_2,\omega_3, \dots, \omega_k)=g_{k+1} $$

\noindent where, in fact, $\alpha(\omega_2,\omega_3, \dots, \omega_k)$ is determined by $g_2,\ldots, g_k$. So there is one choice of the parity of $g_{k+1}$ making that $w_{k+1} = \frac{g_{k+1} - \alpha(\omega_2,\omega_3, \dots, \omega_k)}{2}$ is an integer solution and the opposite one makes it not having an integer solution.
\end{proof}

\medskip
\noindent \textbf{Remark.} \emph{Without the intention of deepening into this question, let us briefly comment the following fact. According to the previous proposition, the set of elements in $\mathcal J(\mathbb Z)$ admitting an iterative square root is ``small''. A precise statement, would be the following. Let us define the following analogue of \emph{natural density}, $d:2^{\mathcal J(\mathbb Z)}\to[0,1]$ given by:}
$$d(X)=\lim_{k\to\infty}\lim_{n\to\infty}\frac{a(n,k)}{(2n+1)\cdot k} $$

\noindent \emph{where $a_{n,k}$ is the number of elements $R_k(1,\omega)$ in $\mathcal R_k(\mathbb Z)$, such that the coefficients of the powers $x^2,\ldots, x^k$ of  $\omega$ are in the interval $[-n,n]$. In this case, the density of the set of elements in $\mathcal J(\mathbb Z)$ admitting an iterative square root is 0.}

Notwithstanding the aforementioned proposition, we may find some necessary conditions for the existence of iterative square roots in $\mathcal J(\mathbb Z)$:

\begin{prop} Let $g=x+g_2x^2+g_3x^3+\ldots\in\mathcal J(\mathbb Z)$.

\begin{itemize}
\item  If $g$ admits an iterative square root $\omega=x+\omega_2x^2+\omega_3x^3+\ldots$, then $g_2,g_3$ are even numbers and
$$8g_4-10g_2(g_3-g_2)-g_2^3\equiv 0\mod 16$$

\noindent and
$$\omega_2=g_2/2,\qquad \omega_3=(g_3-g_2)/2,\qquad\omega_4=(8g_4-10g_2(g_3-g_2)-g_2^3)/16. $$

\item  The two conditions above are also sufficient to ensure the existence of a square root of $R_4(1,g)$.

\end{itemize}

%$g_2\equiv 0\mod 4$, $g_4$ is even.

%\item $g_2\equiv 2\mod 4$, $g_4$ has the opposit parity of $\omega_3$

\end{prop}

\begin{proof} 
\noindent In view of Equation \eqref{eq.ass},  we can see that, if  $(R_4(1,\omega))^2 = (R_4(1,g))$, then the 1-column of $(R_4(1,\omega))^2$ verify
$$\begin{bmatrix}0 \\ 1 \\ 2\omega_2\\ 2\omega_2+2\omega_3\\ 2\omega_4+5\omega_2\omega_3+\omega_2^3 \end{bmatrix} = \begin{bmatrix}
0 \\ 
1 \\ g_2\\g_3\\ g_4 \end{bmatrix} \Rightarrow \begin{cases}
    \omega_2 = g_2 / 2 \\
    \omega_3=(g_3-g_2)/2 \\
\end{cases} $$

Hence, substituting in the last equality, we obtain:

$$
\omega_4=(8g_4-10g_2(g_3-g_2)-g_2^3)/16.
$$
Consequently, since $\omega_4$ is an integer, it must verifies $8g_4-10g_2(g_3-g_2)-g_2^3\equiv 0\mod 16$.
\end{proof}
%2\omega_5+6\omega_2\omega_4+2\omega_2^2\omega_4+\omega_2\omega_3^2+3\omega_3^2+3\omega_2^2\omega_3

%\noindent which is equivalent to 
%$$$\alpha=\frac{f}{f\circ \omega}\cdot (\alpha\circ g). $$
%\end{proof}

%$$ \begin{scriptsize}\left[\begin{array}{lllllllll} 1\\ 0 & 1  \\ 0 &\omega_2 & 1 \\0&\omega_3 & 2\omega_2& 1\\0& \omega_4&2\omega_3+\omega_2^2 & 3\omega_2& 1\\0& \omega_5&2\omega_4+2\omega_2\omega_3 & 3\omega_3+3\omega_2^2& 4\omega_2& 1\\ 0& \omega_6& 2\omega_5+2\omega_2\omega_4+\omega_3^2& 3\omega_4+6\omega_2\omega_3+\omega_2^3& 4\omega_3+6\omega_2^2& 5\omega_2&1 & \\ 0& \omega_7 & 2\omega_6+2\omega_2\omega_5+2\omega_3\omega_4 & 3\omega_5+6\omega_2\omega_4+3\omega_3^2+3\omega_2^2\omega_3 &4\omega_4+12\omega_2\omega_3+4\omega_2^3 &5\omega_3+10\omega_2^2 & 6\omega_2 & 1\\ \vdots & \vdots&\vdots &\vdots & \vdots& \vdots&  \vdots&\vdots  &  \ddots \end{array}\right] \end{scriptsize}$$

In the following, we may find sufficient conditions for the existence of iterative square roots:

\begin{prop} Let $g=x+g_2x^2+g_3x^3+\ldots\mathcal J(\mathbb Z)$ such that, for every $k\geq 2$, $g_k\equiv 0\mod 4$. Then $g$ admits an iterative square root. \end{prop}

\begin{proof} We have to prove by induction the statement that follows:

\medskip

\noindent \textbf{Statement.} $R_k(1,g)$ admits a square root $R_k(1,\omega)$ for some $\omega=x+\omega_2x^2+\omega_3x^3+\ldots$ such that $\omega_2,\ldots,\omega_k$ are even.

\medskip

The case $k = 4$ appears in the previous proposition. Moreover, the induction step $k > 4$ is a consequence of Equations \eqref{eq.heaven2} and \eqref{eq.hell2}. This is due to the fact that, if $g_k \equiv 0\mod 4$, there are even integers $\omega_k$ verifying the equations in Theorem \ref{theo.bf}.

%\noindent The cases $k\geq 4$ appear in the previous result. The induction step is a consequence of Equation \eqref{eq.hell2}.
$ $
\end{proof}

We can see interesting examples formal power series $g$ satisfying the conditions above (at least for their first terms) in OEIS:

\begin{itemize}

\item The sequence A028594 of OEIS is 
$$	1, 4, 12, 16, 28, 24, 48, 4, 60, 52,\ldots$$

\noindent and the corresponding generating function has abviously a square root, since it is the expansion of
$$(\theta_3(q) * \theta_3(q^7) + \theta_2(q) * \theta_2(q^7))^2$$

\noindent in powers of $q$.

\item The sequence A261958 is
$$1, 4, 12, 16, 24, 32, 28, 36, 32, 44,\ldots$$

\noindent and each term counts the number of squares added in each step of certain construction.

\item  The sequence A059992 is
$$ 1, 4, 8, 12, 24, 36, 48, 60, 72, 120, \ldots$$

%$$180, 240, 360, 720, 840,\ldots $$

\noindent and corresponds to the numbers with an increasing number of nonprime divisors.

%\medskip
%\textcolor{red}{COSAS QUE PODEMOS INTENTAR: HACER LAS RAICES CUADRADAS DE ALGUNA $g$ INTERESANTE (COMO $g$ del pascal, de la catalan del articulo de he, DE LOS DERANGEMENTS, OEIS A330972, OEIS A008574, OEIS A087080,...)}

\end{itemize}

\noindent For each of them, it is easy to compute the first coefficients of the corresponding square root using the equations in \ref{theo.bf}. For instance,  the coefficients for the last one are
$$1,2,0,2,0,-14,96,-426,1044,\ldots $$

%%%%%%%%%%%%%%%%%%%%%%%%%%%%%%%%%%%%%%%%%%%%%%%%%%%%%%%%%%%%%%%%%%%%%%%
\section{Roots in \texorpdfstring{$\mathcal R'(Z)$}{R'(Z)}} \label{Section_RZ}

In this section, we will reverse our approach and examine the properties of roots within our Riordan group $\mathcal R'(Z)$. As we will show, the roots in $\mathcal R'(Z)$ depend on the existence of roots in the substitution group $\mathcal J(Z)$ and the existence of certain eigenvectors. For instance, if we have that $\mathcal R'(Z)$ is an URE-group, then so is $\mathcal J(Z)$.

Let us recall that, for any Riordan matrix $R(f,g)$, we have the following statements \cite{SGWW}:

\begin{itemize}

\item \textbf{1FTRM)} For every formal power series $H$ whose sequence of coefficients corresponds to the column vector $\bm v$ we have
\begin{equation}\label{eq.ft}R(f,g)\bm v=\bm w \end{equation}

\noindent where $\bm w$ is the column vector corresponding to the sequence of coefficients of $f\cdot (H\circ g)$.

\item \textbf{Operation Law)} As a consequence, for any other Riordan matrix $R(d,h)$, we have that
\begin{equation}\label{eq.ro} R(d,h)\cdot R(f,g)=R(d\cdot (f\circ h),g\circ h). \end{equation}

\end{itemize}

The previous facts have some consequences in relation to stabilisers. We say that $\bm M$ is in the stabiliser of the column vector $\bm v$ if $\bm v$ is an eigenvector of eigenvalue 1 of a given matrix $\bm M$.  Hence, we have the following:

\begin{thm}\label{theo.e} Let $Z$ be any ring and $R(f,g), R(\alpha,\omega)$ be two elements in $\mathcal R'(Z)$.  The Riordan matrix $R(\alpha,\omega)$ is a root of order $n$ of $R(f,g)$ if and only if $\omega$ is an iterative root of order $n$ of $g$ and the Riordan matrix $R\left(\frac{f}{f \circ \omega},g\right)$ is in the stabiliser of the column vector corresponding to the coefficients of $\alpha$.
\end{thm}

\begin{proof} First, note that $(R(\alpha,\omega))^2 = R(\alpha \cdot (\alpha \circ \omega),\omega^{[2]})$ and $(R(\alpha,\omega))^3 = R(\alpha \cdot (\alpha \circ \omega) \cdot (\alpha\circ \omega^{[2]})),\omega^{[3]})$, and so on. 

In general, if we want $(R(\alpha,\omega))^n=R(f,g)$, as a consequence of Equation \eqref{eq.ro} in Operation Law, then necessary $\omega$ is the unique solution of the equation $\omega^{[n]}=g$.

On the other hand, 
\begin{equation}\label{eq.e1}\alpha\cdot (\alpha\circ \omega)\cdot\ldots\cdot (\alpha \circ \omega^{[n-1]})=f. \end{equation}

\noindent Performing the composition of both sides in the equation by $\omega$ we obtain
$$(\alpha\circ\omega)\cdot (\alpha\circ \omega\circ\omega)\cdot\ldots\cdot \alpha \circ \omega^{[n]}=f\circ\omega$$

\noindent which is equivalent to
\begin{equation}\label{eq.e2}[(\alpha\circ\omega)\cdot (\alpha\circ \omega\circ\omega)\cdot\ldots\cdot \alpha \circ \omega^{[n-1]}]\cdot ( \alpha \circ g)=f\circ\omega. \end{equation}

\noindent Dividing Equation \eqref{eq.e1} and Equation \eqref{eq.e2} we obtain
$$\frac{\alpha}{\alpha\circ g}=\frac{f}{f\circ \omega} $$

\noindent which is equivalent to 
$$\alpha=\frac{f}{f\circ \omega}\cdot (\alpha\circ g).$$

Hence, if $\bm v$ is the column vector corresponding to the coefficients of $\alpha$:
$$
R\left(\frac{f}{f \circ \omega},g\right) \bm v = 1 \cdot \bm v
$$
i.e, $R\left(\frac{f}{f \circ \omega},g\right)$ is in the stabiliser of $\bm v$.
\end{proof}

\medskip

\noindent \textbf{Remark} \emph{According to the previous theorem, to study existence and uniqueness of roots of order $n$ of some element  $R(f,g)\in\mathcal R'(Z)$, it suffices to study}

\begin{itemize}
    \item[(1)]\emph{ existence and uniqueness of solutions of the equation $\omega^{[n]}=g$ in $\mathcal J(Z)$  and}
    \item[(2)] \emph{the existence of eigenvectors $\bm v=[v_0,v_1,v_2,\ldots]^T$ of eigenvalue 1 and with $v_0\neq 0$ of $R\left(\frac{f}{f \circ \omega},g\right)$.}
\end{itemize}  

\emph{We have already done the study of the first problem in this paper. Some discussion about the second one can be found in \cite{LMP}, as  explained in the following section, in the case of $Z$ being a field of characteristic $0$.}

\section{Roots in \texorpdfstring{$\mathcal R'(Z)$}{R'(Z)}, for some particular examples of the ring \texorpdfstring{$Z$}{Z}}\label{Section_RZexamples}

Consequently, based on the previous findings and the established connections between $\mathcal J(Z)$ and $\mathcal R'(Z)$, we may fix the case in which the Riordan group $\mathcal R'(\mathbb K)$ defined over a field of characteristic 0. Thus, we have the following:

\begin{thm} For every field $\mathbb K$ of characteristic 0, $\mathcal R'(\mathbb K)$ is an algebraically complete URE-group.

\end{thm}

\begin{proof} We need to check that, for every $R(f,g)$, there is an unique  $R(\alpha,\omega)$ with the property $(R(\alpha,\omega))^n=R(f,g))$. In view of the last remark in the previous section, the existence and uniqueness required in Statement (1) is guaranteed in Theorem \ref{theo.bf} and the existence and uniqueness required in Statement (2) is a consequence of  Proposition 8, Statement (b) in \cite{LMP}.
\end{proof}

As can be read in \cite{LMP} (proof of Proposition 8), discussing questions about existence of eigenvectors of elements in $\mathcal R'(Z)$ is not easy in the general case (in \cite{LMP} only the case of $Z$ being a field of characteristic 0 is considered but, in an analogue way to what happened before in the present paper, the equations stated are also valid for any other commutative ring with unity $Z$). This can be a question for future research.

At least, it is possible to say the following for the case $Z=\mathbb Z$:

\begin{thm} $\mathcal R'(\mathbb Z)$ is an  URE-group.

\end{thm}

\begin{proof} $\mathcal R'(\mathbb Z)$ is a subgroup of $\mathcal R'(\mathbb Q)$.
\end{proof}

In general, the study of the algebraic structure of $\mathcal R'(Z)$ is much more complicated if $Z$ is not a field of characteristic 0 and we hope that more findings can be unveiled in the future.  In particular, these kind of problems that have appeared here in connection to the square roots of elements in $\mathcal R'(Z)$, require to study how even and odd numbers are distributed among the entries of the matrix. This ideas have already been explored for the Pascal Triangle. It is known that iterations of  the Sierpinski triangle are obtained when we color the odd numbers in matrices of increasing size (see \cite{St}).

\section{Connections to some other problems}\label{SectionConnections}

\subsection{Stabiliser of an infinite column vector}

As we have already showed, powers and roots of elements in $\mathcal R'(Z)$ are related to stabilisers of column vectors. We have the following:

\begin{thm} Let $R(\alpha,\omega),R(f,g)\in\mathcal R'(Z)$ and $\bm v=[1,v_1,v_2\ldots]^T$, with $v_1,v_2,\ldots\in Z$. 

\begin{itemize}

\item If $R(\alpha,\omega)$ is in the stabiliser of $\bm v$, then so is $(R(\alpha,\omega))^n$, for every integer $n\geq 1$. 

\item If $R(f,g)$ is in the stabiliser of $\bm v$, $\mathcal R'(Z)$ is an URE group, and $R(\alpha,\omega)$ is a root of order $n$ of $R(f,g)$, then $R(\alpha,\omega)$ is in the stabiliser of $\bm v$.

\end{itemize}

\end{thm}

\begin{proof} The first statement is immediate. Let us prove the second one.

Suppose that $R(f,g)$ is in the stabiliser of $\bm v$. Let $R(\alpha,\omega)$ be a root of order $n$ of $R(f,g)$. 

Note that, for every $h\in\mathcal J(Z)$, there exists an unique  $d\in (1+xZ[[x]])$ (i.e, $d = 1 + d_1 x + d_2 x^2 + \dots$) such that $R(d,h)$ is in the stabiliser of $\bm v$. This fact can be proved using that $R(d,x)R(1,h)= R(d,h)$. Thus, $R(d,x)R(1,h)\bm v = R(d,h)\bm v = \bm v$, and this equation has an unique solution. Let us call this fact \emph{uniqueness for the first component.}

Hence, there is a unique $\widetilde \alpha$ such that $R(\widetilde\alpha,\omega)$ is in the stabiliser of $\bm v$. Let us consider the Riordan matrix $(R(\widetilde\alpha,\omega))^n=R(\widetilde f,g)$. According to the first statement in the theorem, we have that $R(\widetilde f,g)$ is in the stabiliser of $\bm v$. The \emph{uniqueness in the first component} allows us to conclude that $\widetilde f=f$. 

Finally, note that we have proved that $(R(\widetilde\alpha,\omega))^n=R(f,g)$ but, since $\mathcal R'(Z)$ is an URE-group, then $\widetilde \alpha=\alpha$. So $R(\alpha,\omega)$ is in the stabiliser of $\bm v$.
\end{proof}

\subsection{Conjugacy classes in \texorpdfstring{$\mathcal J(Z)$}{J(Z)}}

The problem of computing roots of order $n$ of elements in $\mathcal J(Z)$ or $\mathcal R'(Z)$ is, of course, related to the problem of computing representatives of the conjugacy classes of those elements. Canonical representatives of elements in $\mathcal J(\mathbb C)$ have been studied in the literature (see, for example, \cite{B,CPT,Mu,Sc}).

For the case $Z=\mathbb C$, it is well known that, for every $\omega\in\mathcal J(\mathbb C)$, there is another element $g=x+a x^n+bx^{2n-1}$ in the same conjugacy class, that is, such that there exists some $h\in\mathcal J(Z)$ such that $\omega=h\circ g\circ h^{-1}$. Since $\mathcal J(\mathbb C)$ is an algebraically complete URE-group, then the conjugacy action by $h$ is a bijection between the roots of $g$ and the roots of $\omega$. So, for some purposes, it is enough to study the roots of elements of the type $x+a x^n+bx^{2n-1}$. This would have been another valid approach to prove some of the results in this paper.

%Let $\mathbb K$ be a field of characteristic 0. For every $g\in\mathcal J(\mathbb K)$ of multiplicity $m$, it is possible to find a canonical representative 
%$$h=x+x^m+bx^{2m-1}. $$

%\noindent such that $g= \varphi^{-1}\phi h\phi\varphi$, for some $\varphi\in\mathcal J(\mathbb K)$. This result, for $\mathbb K=\mathbb C$ (although the proof is valid for any field $\mathbb K$ of characteristic 0), can be found in \cite{Mu}. So, to determine the root of order $n$ of $g$, it suffices to determine the root of order $n$ of $h$.

\subsection{Roots in \texorpdfstring{$\mathcal A(\mathbb K)$}{A(K)}}
Throughout the paper, we have analysed the group $\mathcal J(Z)$, i.e, the set of formal power series whose first coefficient is $1$. Let us consider $\mathcal A(\mathbb K) = x\mathbb{K}[[x]]\setminus x^2 \mathbb{K}[[x]] $, for a field $\mathbb K$, to be the set of formal power series of the type $g_1x+g_2x^2+\ldots$, with $g_1\neq 0$. This set, with the composition operator, is a group (which is not true if $Z$ is a ring but not a field). 

%On the other hand, for a field $\mathbb K$, let us consider now $\mathcal A(\mathbb K)$ to be the set of formal power series of the type $g_1x+g_2x^2+\ldots$, with $g_1\neq 0$. 

If we want $\mathcal A(\mathbb K)$ to be algebraically complete, then a necessary condition is that $\mathbb K$ is algebraically closed: if $\omega=\omega_1x+\omega_2x^2+\ldots$ is a root of order $n$ of $f=f_1x+f_2x^2+\ldots$, then $\omega_1$ is a root of order $n$ of $f_1$.

In that case, we can repeat the argument done in Theorem \ref{theo.bf} to study existence and uniqueness of roots and we would obtain an analogue of Equation \eqref{eq.hell}. The coefficient of $\omega_k$ in this analogue is
$$\omega_1^{n-1}(1+\omega_1^{n-1}+\omega_1^{2(n-1)}+\ldots+\omega_1^{(n-1)^2}).$$

\noindent So, we have the following:

\begin{prop} Let $\mathbb K$ be an algebraically closed field of characteristic 0 and let $g=g_1x+g_2x^2+\ldots \in \mathcal A(\mathbb K)$, such that $g_1$ is not a root of unity. Then, for every $n$, $g$ has a root of order $n$. Moreover, there are as many such roots, in $\mathcal A(\mathbb K)$, as roots of order $n$ of $g_1$ in the field $\mathbb K$.
\end{prop}

But if, in the notation of the previous corollary, $g_1$ is a root of unity, the situation is more complicated, because the coefficient of $\omega_k$ in the analogue of Equation \eqref{eq.hell} can be 0. 

A particular case that we can solve occurs if $g$ is an element of finite order. In this case, $g_1$ is a root of unity, but $g$ is in the same conjugacy class of $g_1x$ (see, for example, \cite{CPT,Sc}). So, we have that:

\begin{coro} Let $\mathbb K$ be an algebraically closed field of characteristic 0 and let $g=g_1x+g_2x^2+\ldots$ be an element of finite order in $\mathcal A(\mathbb K)$. Then, for every $n$, $g$ has a root of order $n$. Moreover, there are as many such roots as roots of order $n$ of $g_1$ in the field $\mathbb K$.
    
\end{coro}

In relation to this, the so-called \textbf{Babbage equation}
$$\omega\circ\omega=x $$

\noindent  has been studied widely in the literature, typically for $\mathbb K=\mathbb R,\mathbb C$ (see, for example \cite{LMP2} and the references therein).

\subsection{Roots in \texorpdfstring{$\mathcal R(\mathbb K)$}{R(K)}}

Let us recall now the definition of Riordan matrix and of Riordan group in current use in the literature:

\begin{defi} \label{Defi_RK} Let $\mathbb K$ be any field. A \textbf{Riordan matrix} with entries in $\mathbb K$ is an  invertible infinite matrix of the type
$$\begin{bmatrix} d_{00}\\ d_{10} & d_{11}\\ d_{20} & d_{21} & d_{22} \\ d_{30} & d_{31} & d_{32} & d_{33} \\ \vdots & \vdots & \vdots & \vdots & \ddots \end{bmatrix},$$

\noindent  such that there exists two formal power series $f\in (\mathbb K[[x]]\setminus x\mathbb K[[x]])$ and $g\in\mathcal A(\mathbb K)$ satisfying that, for every $j\in\mathbb N$, the generating function of the column $[0,\dots,0,d_{jj},d_{j+1,j},\ldots]^T$ is $f\cdot g^j$, that is,
$$f\cdot g^j=d_{jj}x^j+d_{j+1,j}x^{j+1}+d_{j+2,j}x^{j+2}+\ldots$$

\noindent In this case, we denote the corresponding matrix by $R(f,g)$. The \textbf{Riordan group} $\mathcal R(\mathbb K)$ is the set of all the Riordan matrices with entries in $\mathbb K$ with the multiplication of matrices.
\end{defi}

Again, for algebraically closed fields $\mathbb K$ of characteristic 0 it makes sense to try to extend the results appearing here concerning existence and uniqueness of roots.

In general, to decide if a given matrix $R(f,g)$, as explained in the final remark in Section \ref{Section_RZ}, we need to study: (1) the existence and uniqueness of solutions of the equation $\omega^{[n]}=g$ (that we have already discussed in the previous subsection since, this time, $g=g_1x+g_2x^2+\ldots$ for some $g_1$ that may not be 1); and (2) the existence of an eigenvector of eigenvalue $1$ of a given matrix. But the result appearing in \cite{LMP} that we have used above requires $g_1=1$.  The problem of deciding if a Riordan matrix has, in general, an eigenvector or not is more complicated, even in the case $\mathbb K=\mathbb C$ (see \cite{CPT}, where a complete study of this problem is provided).

As in the previous subsection, it is possible to extend the result to elements such that $g_1$ is not a root of unity or such that $R(f,g)$ is an element of finite order. As a particular case, for $\mathbb K=\mathbb R,\mathbb C$, the study of \textbf{involutions} in the Riordan group, that is, Riordan matrices $\bm M$ whose square is the identity matrix have also received some attention. We refer to \cite{CK, CJKS, LMP2, LMP3} and the references therein.

\subsection{Open questions}

To conclude this paper, we outline several unresolved questions arising from this research that we deem significant for future research:

\noindent\textbf{Open Problem 1.} To characterize the set of elements in $\mathcal J(\mathbb Z)$ and in $\mathcal R'(\mathbb Z)$ admitting roots of order $n$.

\noindent\textbf{Open Problem 2.} For $\mathbb K$ being an algebraically closed field of characteristic 0, to decide wether $g=g_1x+g_2x^2+\ldots\in\mathcal A(\mathbb K)$,  has a root of order $n$ in $\mathcal A(\mathbb K)$ remains open.

\noindent\textbf{Open Problem 3.} The previous open problem can be translated to Riordan matrices: for  $\mathbb K$ being an algebraically closed field of characteristic 0, to decide wether $R(f,g)\in\mathcal R(\mathbb K)$ has a root of order $n$ in $\mathcal R(\mathbb K)$ remains open.

%\bibliographystyle{elsarticle-num} 
%\bibliography{cas-refs}


\end{document}